\documentclass[11 pt]{amsart}
\usepackage{epsfig,graphics,amsmath,amssymb,amsthm}%\textheight20cm

\newtheorem{Theorem}{Theorem}[section]

\newtheorem{Corollary}[Theorem]{Corollary}

\newtheorem{Lemma}[Theorem]{Lemma}

\theoremstyle{definition}

\def\PMod{\rm PMod}
\def\Mod{\rm Mod}
\def\trace{\rm trace}

\begin{document}

%\doublespacing

\title[Pseudo--Anosov elements in the mapping class group]
{The number of pseudo--Anosov elements in the mapping class group
of a four--holed sphere}

\author{Ferihe Atalan} \author{Mustafa Korkmaz}
\address{Department of Mathematics, Atilim University,
06836 \newline Ankara, TURKEY}\email{fatalan@atilim.edu.tr}
\address{Department of Mathematics, Middle East Technical University,
06531 \newline Ankara, TURKEY}
\email{korkmaz@metu.edu.tr}
\date{\today}
%\thanks{The author is partially supported by.}
 \pagenumbering{arabic}

\begin{abstract} We compute the growth series
and the growth functions of reducible and pseudo--Anosov elements
of the pure mapping class group of the sphere with four holes with
respect to a certain generating set. We prove that the ratio of
the number of pseudo--Anosov elements to that of all elements in a
ball with center at the identity tends to one as the radius of the
ball tends to infinity.
\end{abstract}

\maketitle
\section{Introduction}
A finitely generated group can be seen as a metric space after
fixing a finite generating set. The metric is the so called word
metric. As is well--known, the mapping class group of a compact
surface is finitely generated, thus a metric space.

The purpose of this note is to prove that, after fixing a certain
set of generators, in a ball centered at the identity in the pure
mapping class group of a four holed sphere (which is a free group
of rank two), almost all elements are pseudo--Anosov. More
precisely, in a ball with center at the identity, the ratio of the
number of pseudo--Anosov elements to the number of all elements
tends to one as the radius of the ball tends to infinity. In fact,
we prove more: We give the growth series of reducible and of
pseudo--Anosov elements with respect to a fixed set of generators.
It turns out that the growth functions of these elements are
rational. This gives a partial answer to Question 3.13  and
verifies Conjecture 3.15 in~\cite{F} in a special case.

\bigskip
\section{Preliminaries}
\bigskip

Let $G$ be a finitely generated group with a finite generating set
$A$, so that every element of $G$ can be written as a product of
elements in $A\cup A^{-1}$. The {\em length} of an element $g\in
G$ (with respect to $A$) is defined as
\[
||g||_A=\min \{ k \, \colon \,  g=a_1a_2\cdots a_k, \, a_i\in
A\cup A^{-1}\}.
\]
The {\em distance} between two elements $g$ and $h$ is defined as
$d_A(g,h)=||h^{-1}g||_A$. The function $d_A$ is a metric on $G$,
called the {\em word metric}. Of course, this metric depends
heavily on the generating set. The choice of different generating
sets give rise to equivalent metrics. We will always fix a finite
generating set $A$ and drop $A$ from the notation.

For a subset $P$ of $G$, the {\em growth series} of $P$ relative
to the generating set $A$ is the formal power series $\sum  c_n
x^n$, where the coefficient $c_n$ of $x^{n}$ is the number of
elements of length $n$ in $P$. The {\em growth function} of $P$ is
the function represented by the growth series. In the mapping
class group, we may take $P$ to be periodic, reducible or
pseudo--Anosov elements.

Let $S$ be a compact connected orientable surface of genus $g$
with $r\geq 0$ holes (= boundary components). The {\em mapping
class group} $\Mod (S)=\Mod (g,r)$ of $S$ is defined as the group
of isotopy classes of orientation--preserving homeomorphisms $S\to
S$. The subgroup $\PMod (g,r)$ of $\Mod (g,r)$ consisting of
isotopy classes of homeomorphisms preserving each boundary
component of $S$ is the {\em pure mapping class group}.

Thurston's classification of surface diffeomorphisms says that for
a mapping class $f$ which is not the identity exactly one of the
followings holds: (1) $f$ is periodic, i.e. $f^m=1$ for some
$m\geq 2$; (2) $f$ is reducible, i.e. there is a (closed)
one--dimensional submanifold $C$ of $S$ such that $f(C)=C$; (3)
$f$ is pseudo--Anosov (Anosov if $S$ is a torus).

It is well known that the mapping class group $\Mod (1,0)$ of a
torus is isomorphic to $SL(2, \mathbb Z)$. The elements of the
group $\Mod (1,0)$ are classified by the traces of the
corresponding matrices; if $f$ is an element of $\Mod (1,0)$, then
it is periodic if $|\trace (f)|< 2$, reducible if $|\trace
(f)|=2$, and Anosov if $|\trace (f)|> 2$ (cf. see \cite{CB}). In
\cite{T}, Takasawa computed the growth series of periodic,
reducible and Anosov elements of $\Mod (1,0)$ and found their
growth functions. He proved that almost all elements of the
mapping class group of the torus are Anosov. That is, with respect
to a certain generating set, the ratio of the number of Anosov
elements to the number of all elements in a ball centered at
 the identity tends to one as the radius of the ball tends to infinity.

Now let $S$ be a sphere with four holes and let $a$ and $b$ be two
distinct nonisotopic simple closed curves on $S$ such that each of
$a$ and $b$ separates $S$ into two pairs of pants and that $a$
intersects $b$ precisely at two points (c.f.
Figure~\ref{lantern}). It is well known that $\PMod (0,4)$ is
isomorphic to the free group $F_2$ and freely generated by the
Dehn twists $t_a$ and $t_b$ about $a$ and $b$ respectively. We
will always take this generating set below.

\bigskip
\section{The number of reducible and pseudo--Anosov elements
in the mapping class group $\PMod(0,4)$}
\bigskip

\subsection{Counting certain elements in the free group of rank two.}
 We begin by counting certain type of elements in the free group of rank
 two. Let $F_2$ be the free group of rank two freely generated by $\{ \alpha,
 \beta\}$. We fix this set of generators throughout this subsection.

The next lemma is elementary and is easy to prove.

\begin{Lemma} \label{lemma:free}
The growth series of $F_2$ is
\begin{eqnarray*}
h(x)&=& 1+4x+4\cdot
3x^{2}+4\cdot3^{2}x^{3}+\cdots+4\cdot3^{n-1}x^{n}+\cdots .
\end{eqnarray*}
\end{Lemma}

\smallskip

For an element $\gamma\in F_2$, let $C({\gamma}, n)$ denote the
set of elements in $F_2$ of length $n$ of the form $w\gamma^k
w^{-1}$, where $k$ is an integer and $w\in F_2$. Let $|C({\gamma},
n)|$ denote the cardinality of $C({\gamma}, n)$.

\smallskip

\begin{Lemma} \label{lemma:odd}
\begin{enumerate}
  \item  If $w\alpha^k w^{-1}$ and $v\alpha^l v^{-1}$ are reduced, then $w\alpha^k w^{-1}=v\alpha^l v^{-1}$
    if and only if $w=v$ and $k=l$.
  \item For each nonnegative integer $r$, $|C(\alpha,2r+1)|=| C(\alpha,2r+2)|=|
    C(\beta,2r+1)|=| C(\beta,2r+2)|=2 \cdot 3^{r}$.
  \item For each nonnegative integer $r$, $|C(\alpha\beta ,2r+1)|=0$ and
    $|C(\alpha\beta,2r+2)|= 4 \cdot 3^{r}$.
\end{enumerate}
\end{Lemma}

\begin{proof}
If $w\alpha^k w^{-1}=v\alpha^l v^{-1}$ then
$\alpha^{k-l}=w^{-1}v\alpha^l v^{-1}w \alpha^{-l}$, a commutator.
Hence, $k=l$. Now, by looking at the lengths of each side of
$\alpha^k= w^{-1}v\alpha^k v^{-1}w$, we deduce that $w=v$. The
converse is clear, proving~(1).

Define a function $\phi\colon C(\alpha,2r+1)\to C(\alpha,2r+2)$
by
\[
\phi(w\alpha^kw^{-1})= \left\{
 \begin{array}{ll}
 w\alpha^{k+1}w^{-1}, & \mbox{if $k>0$}\\
 w\alpha^{k-1}w^{-1}, & \mbox{if $k<0$},
 \end{array}
 \right.
\]
 where
$w\alpha^kw^{-1}$ is reduced. Clearly, the function $\phi$ is
onto. It follows from~(1) that it is also one--to--one. Consider
also the automorphism $\psi \colon F_2\to F_2$ given by $\psi
(\alpha)=\beta$ and $\psi (\beta)=\alpha$. The map $\psi$ is an
isometry and $\psi ( C(\alpha, n) )= C(\beta, n)$. Thus, the first
three equalities in~(2) are proved. In order to complete the proof
of~(2), we show $|C(\alpha,2r+1)|=2 \cdot 3^{r}$. The proof of
this claim is by induction on $r$.

Note that if $k$ is even then the length of $w\alpha^kw^{-1}$ is
even for any $w\in F_2$. Hence, $C(\alpha,2r+1)$ contains the
conjugates of odd powers of $\alpha$. Note also that if
$w\alpha^kw^{-1}$ is a reduced word of length $n$, then $-n\leq k
\leq n$.

The set $C(\alpha,1)$ contains only two elements, $\alpha$ and $\alpha^{-1}$.
Hence, the claim holds in the case $r=0$.

Assume that $|C(\alpha,2r+1)|=2\cdot 3^r$. Define a function
$\varphi$ from $C(\alpha,2r+1)$ to the subsets of $C(\alpha,2r+3)$
as follows:\\
   $\bullet \quad  \varphi (\alpha^{2r+1} )=\{ \alpha^{2r+3},
     \beta\alpha^{2r+1}\beta^{-1},\beta^{-1}\alpha^{2r+1}\beta \}$;\\
   $\bullet \quad  \varphi (\alpha^{-(2r+1)})=\{ \alpha^{-(2r+3)},
        \beta\alpha^{-(2r+1)}\beta^{-1}, \beta^{-1}\alpha^{-(2r+1)}\beta \}$;\\
  $\bullet \quad  \varphi ( \alpha w \alpha^{-1})=\{ \alpha^2 w \alpha^{-2} ,
        \beta\alpha w \alpha^{-1}\beta^{-1},
           \beta^{-1}\alpha w \alpha^{-1} \beta \}$;\\
   $\bullet \quad  \varphi ( \alpha^{-1} w \alpha)=\{ \alpha^{-2}w \alpha^2 ,
            \beta\alpha^{-1} w \alpha\beta^{-1},
           \beta^{-1}\alpha^{-1} w \alpha \beta \}$;\\
   $\bullet \quad  \varphi ( \beta w \beta^{-1})=\{ \beta^2 w \beta^{-2} ,
   \alpha\beta w \beta^{-1}\alpha^{-1},
           \alpha^{-1}\beta w \beta^{-1} \alpha \}$;\\
   $\bullet \quad  \varphi ( \beta^{-1} w \beta)=\{ \beta^{-2} w \beta^2 ,
        \alpha\beta^{-1} w \beta\alpha^{-1},
           \alpha^{-1}\beta^{-1} w \beta \alpha \}$.

It is easy to check that the set
 \[ \{ \varphi(x) \, : \, x\in C(\alpha,2r+1) \} \]
is a partition of $C(\alpha,2r+3)$. That is, elements of this set
are pairwise disjoint and their union is equal to
$C(\alpha,2r+3)$. We deduce from this that
$|C(\alpha,2r+3)|=3|C(\alpha,2r+1)|=2\cdot 3^{r+1}$, completing
the proof of~(2).

It is clear that $|C(\alpha\beta, 2r+1)|=0$ for all $r\geq 0$.
Note that for any $w\in F_2$, the word length of
$w(\alpha\beta)^kw^{-1}$ is at least $2|k|$. That is, the set
$C(\alpha\beta ,2r+2)$ does not contain any conjugate of
$(\alpha\beta)^k$ for $|k|>r+1$.

The element $(\beta\alpha)^n$ is conjugate to $(\alpha\beta)^n$
and any element in $C(\alpha\beta ,2r+2)$ is of the form
$w(\alpha\beta)^nw^{-1}$ or $w(\beta\alpha)^nw^{-1}$ for some
$w\in F_2$ with $||w||=r+1-n$. Hence, we will only consider the
(reduced) words in these two forms.

The only conjugates of $(\alpha\beta)^k$ for $|k|=r+1$ contained
in $C(\alpha\beta ,2r+2)$ are elements of
\begin{eqnarray*}
A_{r+1}=\{ (\alpha\beta)^{r+1},(\beta\alpha)^{r+1},
(\alpha\beta)^{-(r+1)}, (\beta\alpha)^{-(r+1)} \}.
\end{eqnarray*}
All other elements of $C(\alpha\beta ,2r+2)$ are conjugates of
$(\alpha\beta)^{k}$ for $|k|\leq r$, hence they are conjugates of
elements of $C(\alpha\beta ,2r)$.

Consider the subset of  $C(\alpha\beta ,2r)$ consisting of the
conjugates of $(\alpha\beta)^{\pm r}$. They form the set
\begin{eqnarray*}
  A_{r}=\{ (\alpha\beta)^{r},(\beta\alpha)^{r},(\alpha\beta)^{-r},
  (\beta\alpha)^{-r} \}.
\end{eqnarray*}
Each element of $A_r$ gives rise two elements of length $2r+2$ by
conjugation. For instance, one may conjugate $(\alpha\beta)^{r}$
only with $\alpha$ and $\beta^{-1}$ in order to get an element of
length $2r+2$. Therefore, there are eight such elements in
$C(\alpha\beta ,2r+2)$.

The elements of the difference $C(\alpha\beta ,2r)- A_r$ are of
the form $\alpha w\alpha^{-1}$, $\alpha^{-1} w\alpha$, $\beta
w\beta^{-1}$ or $\beta^{-1} w\beta$. The number of such elements
is $|C(\alpha\beta ,2r)|-4$ and each gives rise to three elements
of length $2r+2$ by conjugation (if there is cancellation, we do
not need to take them).

It follows that
\[
|C(\alpha\beta ,2r+2)| = 4 + 8 + 3\left( |C(\alpha\beta
,2r)|-4\right)=  3 |C(\alpha\beta ,2r)|.
\]
Now, (3) follows from the fact  that $C(\alpha\beta ,2)$ consists
of four elements; namely
\[
C(\alpha\beta ,2)=\{ \alpha\beta, \beta\alpha, (\alpha\beta)^{-1},
(\beta\alpha)^{-1} \}.
\]

This finishes the proof of the lemma.
\end{proof}

\smallskip

\begin{Corollary} \label{cor:all}
The number of elements of length $n$ conjugate to a power of
$\alpha$, $\beta$ or $\alpha\beta$ is $4\cdot 3^r$ if $n=2r+1$ and
$8\cdot 3^r$ if $n=2r+2$ $(r\geq 0)$.
\end{Corollary}
\begin{proof}
The set of elements of length $n$ conjugate to the given elements
is $C(\alpha ,2r+1)\cup C(\beta ,2r+1)$ if $n=2r+1$ and $C(\alpha
,2r+2)\cup C(\beta ,2r+2) \cup C(\alpha\beta ,2r+2)$ if $n=2r+2$.
These sets are pairwise disjoint. The result now follows from
Lemma~\ref{lemma:odd}.
\end{proof}

\smallskip

\subsection{The mapping class group $\PMod (0,4)$}

Since $\PMod (0,4)$ is isomorphic to $F_2$, there are no periodic
elements in $\PMod (0,4)$. Elements of $\PMod (0,4)$ different
from the identity are either reducible or pseudo--Anosov. In this
section, we compute the growth series and the growth functions of
these elements in $\PMod (0,4)$.

Let $S$ be a sphere with four holes. A simple closed curve $a$ on
$S$ is called {\em trivial} if either it bounds a disc or it is
parallel to a boundary component. Otherwise, it is called
nontrivial.

Let us fix two nontrivial simple closed curves $a$ and $b$ on $S$
intersecting transversely twice as in Figure~\ref{lantern}. It is
well known that the Dehn twists $t_a$ and $t_b$ generate the group
$\PMod (0,4)$ freely. By the lantern relation, there is a unique
simple closed curve $c$ on $S$ separating $S$ into two pairs of
pants and intersecting both $a$ and $b$ twice such that the Dehn
twists $t_a,t_b$ and $t_c$ satisfy $t_at_bt_c=1$ (c.f.
Figure~\ref{lantern}). Thus, we have $t_c=(t_at_b)^{-1}$, and
hence conjugates of powers $t_a$, $t_b$ and $t_at_b$ are
reducible. In fact, they are the only reducible elements in $\PMod
(0,4)$.

\begin{figure}[hbt]
    \begin{center}
        \includegraphics[width=4cm]{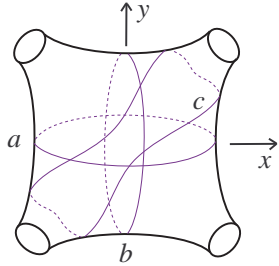}
        \caption {The Dehn twists about $a,b,c$ satisfy $t_at_bt_c=1$ in
        $\PMod (0,4)$ by the lantern relation.}
        \label{lantern}
    \end{center}
\end{figure}

\smallskip

\begin{Lemma} \label{lemma:reduce}
The reducible elements of $\PMod (0,4)$ consist of conjugates of
nonzero powers of $t_a$, $t_b$ and $t_at_b$.
\end{Lemma}

\begin{proof} Let $f$ be a reducible element $\PMod (0,4)$.
Then $F(d)=d$ for some nontrivial simple closed curve $d$ and
$F\in f$. Thus, $t_df=ft_d$, since $ft_{d}f^{-1}=t_{F(d)}=t_{d}$.
Since $\PMod (0,4)$ is a nonabelian free group and $t_d$ can be
completed to a free basis of $\PMod (0,4)$, we conclude that
$f=t_{d}^{k}$ for some nonzero integer $k$.

It follows from the classification of simple closed curves on $S$
(c.f. see \cite{I}) that there is a homeomorphism $H:S\to S$
preserving each boundary component of $S$ such that $H(d)\in \{
a,b,c\}$.

Let $h$ denote the isotopy class of $H$ in $\PMod (0,4)$. If
$H(d)=a$ then $f=t_{d}^{k}=h^{-1}t_{a}^{k} h$, if $H(d)=b$ then
$f=h^{-1}t_{b}^{k} h$, and if $H(d)=c$ then $f=h^{-1}t_{c}^{k} h=
h^{-1}  (t_at_b)^{-k} h$, proving the lemma.
\end{proof}

\smallskip

We are now ready to state and prove the main result of this paper.

\smallskip

\begin{Theorem}
With respect to the generating set $\{ t_a, t_b\}$ of $\PMod
(0,4)$,
\begin{enumerate}
    \item \label{reduce} the growth series of reducible elements is
                \begin{eqnarray*}
                    r(x)&=&
                    4(x+3x^{3}+3^{2}x^{5}+3^{3}x^{7}+\cdots+3^{r}x^{2r+1}+\cdots)\\
                    & & \hspace*{0.3cm} +8(x^2+3x^{4}
                    +3^{2}x^{6}+3^{3}x^{8}+\cdots+3^{r}x^{2r+2}+\cdots).
            \end{eqnarray*}
                    Hence, the growth function of reducible elements is
        \begin{eqnarray*}
        r(x)=\frac{4x+8x^2}{1-3x^2}\,.
        \end{eqnarray*}
    \item \label{growthp-a} the growth series of pseudo--Anosov elements is
         \begin{eqnarray*}
        4 \sum_{r=0}^{\infty} 3^r(3^{r+1}-2)x^{2r+2} +  4\sum_{r=1}^{\infty}
        3^r(3^{r}-1)x^{2r+1}
       \end{eqnarray*}
       and the growth function of pseudo--Anosov elements is
        \begin{eqnarray*}
                p(x)= \frac{4x^{2}(1+3x)}{(1-3x)(1-3x^2)}.
        \end{eqnarray*}
    \item \label{limitp-a}
            if $p_n$ and $h_n$ denote the number of pseudo--Anosov and
            all elements of length at most $n$ respectively, then we have
     \begin{eqnarray*}
       \lim_{n\to\infty}\frac{p_n}{h_n}=1 \,.
     \end{eqnarray*}
\end{enumerate}

\end{Theorem}

\begin{proof}
By Lemma~\ref{lemma:reduce}, reducible elements in $\PMod (0,4)$
are conjugates of nonzero powers of $t_a, t_b$ and $t_at_b$. By
Corollary~\ref{cor:all}, the number of such elements of length
$n>0$ in $\PMod (0,4)$ is $4\cdot 3^r$ if $n=2r+1$ and $8\cdot
3^r$ if $n=2r+2$.

Therefore the growth series of reducible elements is
\begin{eqnarray*}
r(x)&=&4x+4\cdot 3x^{3}+4\cdot 3^{2}x^{5}+4\cdot 3^{3}x^{7}+\cdots
+4\cdot 3^{r}x^{2r+1}+\cdots \\
& & \hspace*{0.3cm}+8 x^{2}+8\cdot
3x^{4}+8\cdot 3^{2}x^{6}+8\cdot
3^{3}x^{8}+\cdots+8\cdot 3^{r}x^{2r+2}+\cdots \\
&=&(4x+8x^2)(1+3x^{2}+3^{2}x^{4}+3^{3}x^{6}+\cdots+3^{r}x^{2r}+\cdots).\\
\end{eqnarray*}
It follows that the growth function is given by
\[
r(x)= \frac{4x+8x^{2}}{1-3x^{2}}.
\]
This proves (\ref{reduce}).

The growth series and the growth function of all elements are
\begin{eqnarray*}
h(x)&=& 1+4x+4\cdot 3x^{2}+4\cdot 3^{2}x^{3}+\cdots+4\cdot
3^{n-1}x^{n}+\cdots\\
    &=& \frac{1+x}{1-3x}\, .
\end{eqnarray*}
The growth series of pseudo--Anosov elements follows from this and
(1). The growth function of pseudo--Anosov elements is
\begin{eqnarray*}
p(x)&=&h(x)-1-r(x)\\
    &=& \frac{4x}{1-3x}-\frac{4x+8x^{2}}{1-3x^{2}} \\
    &=& \frac{4x^{2}(1+3x)}{(1-3x)(1-3x^{2})}.
\end{eqnarray*}
This proves (\ref{growthp-a}).

Let $r_{n}$ denote number of reducible elements of length at most $n$.
By (\ref{reduce}), we have
\begin{eqnarray*}
r_n &=& 4(1+3+3^2+\cdots +3^r)+8(1+3+3^2+\cdots +3^{r-1})\\
    &=& 10\cdot 3^{r}-6
\end{eqnarray*}
if $n=2r+1$ and
\begin{eqnarray*}
r_n &=& 4(1+3+3^2+\cdots +3^{r-1})+8(1+3+3^2+\cdots +3^{r-1})\\
    &=& 2\cdot 3^{r+1}-6
\end{eqnarray*}
if $n=2r$. By Lemma~\ref{lemma:free}, we get
\begin{eqnarray*}
h_n &=& 1+4(1+3+3^2+\cdots +3^{n-1})\\
    &=& 2\cdot 3^{n}-1.
\end{eqnarray*}

It follows that
 \begin{eqnarray*}
       \lim_{n\to\infty}\frac{r_n}{h_n}=0 \,.
     \end{eqnarray*}
Since $p_n=h_n-r_n-1$, the proof of (\ref{limitp-a}) follows.
\end{proof}

\bigskip
\noindent
\subsection{A little more} Let $\imath$ (resp. $\jmath$) denote the isotopy
class of the rotation about the $x$--axis (resp. $y$--axis) by
$\pi$. (We assume that the surface lie in the three space and is
invariant under these rotations, as in Figure~\ref{lantern}.) Let
$\Gamma$ denote the subgroup of the mapping class group
$\Mod(0,4)$ generated by $\PMod (0,4)$, $\imath$ and $\jmath$.
Then $\Gamma$ is isomorphic to $\PMod(0,4)\times {\mathbb
Z}_2\times {\mathbb Z}_2$, and is of index $6$ in $\Mod(0,4)$.

Since $\imath,\jmath$ and $\imath\jmath$ preserve each nonboundary
parallel simple closed curve up to isotopy, it can be shown that
an element $f$ in $\PMod(0,4)$ is pseudo--Anosov if and only if
$f\imath,f\jmath$ and $f\imath\jmath$ are pseudo--Anosov. It
follows that, with respect to the generating set $\{
t_a,t_b,\imath,\jmath\}$ of $\Gamma$, the ratio of the number of
pseudo--Anosov elements to that of all elements in a ball of
radius $n$ centered at the identity tends to one as $n$ tends to
infinity. It would be good to extend this result to $\Mod(0,4)$
and to all $\Mod(0,n)$.

\smallskip

\noindent {\bf Acknowledgement.} The writing of this paper was
finished while the second author was visiting University of
Chicago in January 2008. The second author thanks to Benson Farb
for his interest in this work.


\bigskip
\begin{thebibliography}{1}
\bibitem{CB}  Andrew J. Casson,  Steven A. Bleiler,
{\em Automorphisms of Surfaces After Nielsen and Thurston},
Cambridge University Press, 1988.

\bibitem{F} Benson Farb,
{\em Some problems on mapping class groups and moduli space}, in
``Problems on Mapping Class Groups and Related Topics", ed. by B.
Farb, Proc. Symp. Pure and Applied Math., Vol. 74, pp. 11--55.

\bibitem{H}  Pierre de la Harpe, {\em Topics in Geometric Group Theory},
The University  of Chicago Press, Chicago, 2000.

\bibitem{I} Nikolai V. Ivanov,  {\em Automorphisms of Teichm\"uller modular
groups}, Topology and geometry---Rohlin Seminar, 199--270, Lecture
Notes in Math., 1346, Springer, Berlin, 1988.

\bibitem{T} Mitsuhiko Takasawa, {\em Enumeration of Mapping Classes for the
Torus}, Geometriae Dedicata {\bf 85}, (2001), 11--19.


\end{thebibliography}
\end{document}